\documentclass[12pt,reqno]{amsart}
\usepackage{amsmath,amsthm,amscd,amsfonts,amssymb,enumerate}
\usepackage{graphicx,color}
\usepackage{multirow}
\usepackage{mathrsfs, cite, url}
\usepackage[colorlinks=true,allcolors=blue]{hyperref}

\newtheorem{thm}{Theorem}[section]
\newtheorem{cor}[thm]{Corollary}
\newtheorem{lem}[thm]{Lemma}

\theoremstyle{definition}

\theoremstyle{remark}

\numberwithin{equation}{section}

\begin{document}

\title[Variant Euler harmonic  sums]{Four types of variant Euler harmonic sums}
\author[N. Batir and J. Choi]{Necdet Batir and Junesang Choi$^{*}$}
\address{Necdet Batir: Department of  Mathematics, Nev{\c{s}}ehir hbv University, Nev{\c{s}}ehir, 50300 Turkey}
\email{nbatir@hotmail.com}
\address{Junesang Choi: Department of Mathematics, Dongguk University,
Gyeongju 38066, Republic of Korea}
\email{junesang@dongguk.ac.kr}

\thanks{$^{*}$ Corresponding author}


\begin{abstract} We aim to investigate the four types of variant Euler harmonic sums.
Also, as corollaries, we provide particular examples of our core findings, some of whose further instances are evaluated in terms of basic and well-known functions as well as certain mathematical constants. We explore relevant linkages between our results and those of other previously established studies. An examination of a specific case of one result shows a relationship to series involving zeta functions, which is also a popular area of research.
\end{abstract}
\maketitle

\section{Introduction}\label{Sec-IP}

Certain required functions, numbers and notations are recalled and given. The gamma function $\Gamma$ is given by
\begin{equation}\label{GF}
 \Gamma(z)=\int_{0}^\infty\, t^{z-1}e^{-t}\, dt \quad (\Re(z)>0).
\end{equation}
The  Beta function  $B(\mu,\, \nu)$ is given as follows (see, e.g., \cite[p. 8]{Sr-Ch-12}):
 \begin{equation}\label{beta}
  B(\mu,\, \nu) = \left\{ \aligned & \int_0^1 \, t^{\mu -1} (1-t)^{\nu -1} \, dt
    \quad (\Re(\mu)>0, \,\, \Re(\nu)>0) \\
    &\frac{\Gamma (\mu) \, \Gamma (\nu)}{\Gamma (\mu+ \nu)}
   \hskip 23mm  \left(\mu,\, \nu
      \in \mathbb{C}\setminus {\mathbb Z}_{\leqslant 0} \right). \endaligned \right.
\end{equation}

The psi (or digamma) function $\psi$ is defined by
\begin{equation}\label{Psi-F}
  \psi (z): = \frac{d}{dz} \log  \Gamma(z) = \frac{\Gamma'(z)}{\Gamma(z)} \quad \left(z \in \mathbb{C} \setminus \mathbb{Z}_{\leqslant 0}\right).
\end{equation}
The polygamma function $\psi ^{\left( k\right) }(z)$ is  defined by
\begin{equation}\label{polygamma}
\aligned
\psi ^{\left( k\right) }(z)&:=\frac{d^{k}}{dz^{k}}\{\psi (z)\}=\left(
-1\right)^{k+1}k!\,\sum_{r=0}^{\infty }\frac{1}{\left( r+z\right)^{k+1}}\\
&=\left(-1\right)^{k+1}k!\,\zeta(k+1,z) \quad \left(k \in \mathbb{N},\,\, z \in \mathbb{C} \setminus \mathbb{Z}_{\leqslant 0}\right),
\endaligned
\end{equation}
where $\psi ^{\left(0\right) }(z)=\psi(z)$, and $\zeta(s,z)$ is the generalized (or Hurwitz) zeta function defined by
\begin{equation}\label{GZF}
  \zeta(s,z) = \sum_{j=0}^{\infty}\,\frac{1}{(j+z)^s} \quad \left(\Re(s)>1,\, z \in \mathbb{C} \setminus \mathbb{Z}_{\leqslant 0}\right),
\end{equation}
and $\zeta(s,1) =:\zeta(s)$ is the Riemann zeta function.
 It has the recurrence
\begin{equation}\label{Polyg-rr}
  \psi ^{\left( k\right) }(z+1)=\psi ^{\left( k\right) }(z)+\frac{\left(
-1\right)^{k}k!}{z^{k+1}} \quad \left(k \in \mathbb{Z}_{\geqslant 0}\right).
\end{equation}
Here and in the following, let $\mathbb{C}$,  $\mathbb{R}$, $\mathbb{R}_{>0}$,  $\mathbb{Z}$, and $\mathbb{N}$ denote the sets of complex numbers, real numbers, positive real numbers,   integers, and positive integers, respectively. Also let $\mathbb{Z}_{\geqslant \ell}$ and  $\mathbb{Z}_{\leqslant \ell}$ denote the sets of integers greater than or equal to $\ell$ and less than or equal to $\ell$ for some $\ell \in \mathbb{Z}$.
For more properties and formulas of the above functions, one may refer to
  \cite[Sections 1.1, 1.3 and 2.2, 2.3]{Sr-Ch-12}.

The generalized harmonic numbers $H_n^{(s)}$  of order $s$ are defined by
\begin{equation}\label{GHN}
 H_n^{(s)}:=\sum\limits_{k=1}^{n}\frac{1}{k^s} \quad \left(n \in \mathbb{N},\, s \in \mathbb{C}\right),
\end{equation}
where $H_n^{(1)}=:H_n$ $(n \in \mathbb{N})$ are harmonic numbers, and   $H_0^{(s)}=0=H_0$.
Here and elsewhere, an empty sum is assumed to be nil. The following relations are recalled:
  \begin{equation}\label{e:1}
   H_n=\gamma+\psi(n+1)  \quad \left(n \in \mathbb{Z}_{\geqslant 0}  \right),
  \end{equation}
where $\gamma$ is the Euler-Mascheroni constant (see, e.g., \cite[Section 1.2]{Sr-Ch-12});
\begin{equation}\label{e:2}
H_n^{(m+1)}=\zeta(m+1)+\frac{(-1)^m}{m!}\psi^{(m)}(n+1)\quad \left(m\in \mathbb{N},\, n \in \mathbb{Z}_{\geqslant 0}\right)
\end{equation}
(see, e.g.,  \cite[Eq. (1.25)]{Alzer} and \eqref{polygamma}).
Equations  \eqref{e:1} and \eqref{e:2} are utilized to define
extended harmonic numbers $H_\eta^{(m)}$ of order $m \in \mathbb{N}$
  with index $\eta \in \mathbb{C} \setminus \mathbb{Z}_{\leqslant -1}$
 by  (see \cite{SofoSri})
\begin{equation}\label{G-HN}
 H_\eta^{(m)}:=\left\{
                   \begin{array}{ll}
       \gamma + \psi (\eta+1)               & (m=1), \\
      \zeta (m) +  \frac{(-1)^{m-1}}{(m-1)!} \, \psi^{(m-1)} (\eta+1)  & (m \in \mathbb{Z}_{\geqslant 2}).
                   \end{array}
                 \right.
\end{equation}

A generalized binomial coefficient  $\binom{s}{t}$ ($s,\, t\in\mathbb{C}$) is defined, in terms of the gamma functions, by
\begin{equation}\label{GBC}
\binom{s}{t}:=\frac{\Gamma(s+1)}{\Gamma(t+1)\Gamma(s-t+1)}\quad (s,\, t\in\mathbb{C}).
\end{equation}

Euler found the following identity in 1775, and it has a lengthy history (see, e.g., \cite[p. 252 et seq.]{Berndt}):
\begin{equation}\label{e:3}
\sum_{k=1}^{\infty}\frac{H_k}{(k+1)^2}= \frac{1}{2}\, \sum_{k=1}^{\infty}\frac{H_k}{k^2} =  \zeta(3).
\end{equation}
The identity \eqref{e:3} is a particular instance of the following more generalized Euler sum (see, e.g., \cite{Borwein}, \cite{Choi}, \cite{FlajSalv}, \cite{Sofo-2019}):
\begin{equation} \label{e:5}
2\sum_{k=1}^{\infty}\frac{H_k}{(k+1)^m}=m\,\zeta(m+1)-\sum_{k=1}^{m-2}\,\zeta(k+1)\zeta(m-k) \quad \left(m \in \mathbb{Z}_{\geqslant 2}\right),
\end{equation}
or, equivalently,
\begin{equation} \label{e:5-a}
2\sum_{k=1}^{\infty}\frac{H_k}{k^m}=(m+2)\,\zeta(m+1)-\sum_{k=1}^{m-2}\,\zeta(k+1)\zeta(m-k) \quad \left(m \in \mathbb{Z}_{\geqslant 2}\right).
\end{equation}
During his interaction with Goldbach starting in 1742, Euler initiated a series of investigations for the \emph{linear} harmonic sums \eqref{Euler-sum} (see, e.g., \cite{Choi, FlajSalv}):
   \begin{equation}\label{Euler-sum}
     \mathtt{S}(p,q):=   \sum_{n = 1}^{\infty}\,\frac{H_n^{(p)}}{n^q} \quad \left(p \in \mathbb{N},\, q \in \mathbb{Z}_{\geqslant 2}\right).
   \end{equation}
Euler's study, which Nielsen completed in 1906 (see \cite{Niel-65}), revealed that the linear harmonic sums in \eqref{Euler-sum} may be determined in the following situations:
 $p=1$; $p=q$; $p+q$ odd; $p+q$ even, but  with only the pair $(p,q)$ being the set $\{(2,4),\, (4.2)\}$.
Of these particular instances, in the ones with $p \ne q$, if $ \mathtt{S}(p,q)$ is determined,
then $\mathtt{S}(q,p)$ may be evaluated by means of the symmetry relation
  \begin{equation}\label{Euler-sum-Symm}
     \mathtt{S}(p,q) +\mathtt{S}(q,p)= \zeta(p)\,\zeta(q) + \zeta (p+q)
   \end{equation}
and  vice versa.

The \emph{nonlinear} harmonic sums include  at least two (generalized) harmonic number products.
Put $R=\left(r_1,\,\ldots,\,r_\ell\right)$ be a partition of an integer $r$ into $\ell$ summands,
so that $r=r_1+\cdots+r_\ell$ and $r_1 \leq r_2 \leq \cdots \leq r_\ell$. The Euler sum of index $R,q$ is
defined by
 \begin{equation}\label{Non-Euler-sum}
  \mathtt{S}(R;q):= \sum_{n=1}^{\infty}\, \frac{H_n^{(r_1)}\,H_n^{(r_2)} \cdots H_n^{(r_\ell)}}{n^q},
 \end{equation}
where the quantity $q+r_1+\cdots+r_\ell$ is called the weight, the quantity $\ell$ is the degree.
In partitions, repeating summands are represented by powers for brevity, for instance,
\begin{equation*}
  \mathtt{S}(1^2,2^3,7;q)= \mathtt{S}(1,1,2,2,2,7;q)=\sum_{n=1}^{\infty}\,\frac{H_n^2\, \big\{H_n^{(2)}\big\}^3\,H_n^{(7)}}{n^q}.
\end{equation*}

In the vast mathematical literature, many researchers have undertaken research on Euler, Euler-type sums,
and other versions of these sums using a variety of techniques
(see, e.g.,
\cite{Alzer}, \cite{Basu}, \cite{1}, \cite{Borwein}, \cite{Chavan}, \cite{Choi}, \cite{FlajSalv}, \cite{Freitas}, \cite{Quan}, \cite{Sofo-2022}, \cite{Sofo-2012}, \cite{Sofo-2019}, \cite{SofoSri}, \cite{Xu}  and the references therein).

The main purpose of this paper is to explore the following variants of the Euler harmonic sums: For $n\in\mathbb{Z}_{\geqslant 0}$, $m\in\mathbb{N}$, and
   $p \in \mathbb{C}\setminus \mathbb{Z}_{\leqslant -1}$,
\begin{equation*}
\sum_{k=1}^{\infty}\frac{H_k}{(n+k+1)^m\binom{n+k}{k}}, \quad \sum_{k=1}^{\infty}\frac{H_k^2-H_k^{(2)}}{(n+k+1)^m\binom{n+k}{k}},
\end{equation*}
\begin{equation*}
\sum_{k=1}^{\infty}\frac{H_k}{k(p+n+k)^m\binom{n+k}{k}} \quad \mbox{and}\quad \sum_{k=1}^{\infty}\frac{H_k^2-H_k^{(2)}}{k(p+n+k)^m\binom{n+k}{k}}.
\end{equation*}
In addition, as corollaries, we present specific cases of our primary discoveries, some of whose further particular instances are assessed in terms of elementary and well-known functions as well as certain mathematical constants. We discuss pertinent connections between our findings and those of other known ones. Investigation of a particular instance of one result reveals a connection to series involving zeta functions, which is also an interesting and useful research topic.

\section{Preliminary results}\label{PR}

The following lemma recalls some required properties for the gamma, psi and polygamma functions for easy reference.

 \vskip 3mm
\begin{lem}\label{Lem-A} The followings hold true:
\begin{itemize}
 \item[(i)]
 \begin{equation}\label{Lem-A-eq1}
   \Gamma (z) \Gamma (1-z)= \frac{\pi}{\sin (\pi z)} \quad \left(z \in \mathbb{C} \setminus \mathbb{Z}\right).
 \end{equation}

\item[(ii)] $\Gamma (z)$ and $\psi (z)$ are meromorphic functions on the whole complex $z$-plane with simple poles
 at $z=-k$ $\left(k \in \mathbb{Z}_{\geqslant 0}\right)$ with their respective residues given by
 \begin{equation}\label{Gamma-Res}
   \operatornamewithlimits{Res}_{z=-k}\,\Gamma (z) = \lim\limits_{z\to-k}(z+k)\Gamma(z)=\frac{(-1)^k}{k!}
     \quad \left(k \in \mathbb{Z}_{\geqslant 0} \right)
 \end{equation}
and
\begin{equation}\label{psi-Res}
   \operatornamewithlimits{Res}_{z=-k}\,\psi (z) = \lim\limits_{z\to-k}(z+k)\psi(z)=-1
     \quad \left(k \in \mathbb{Z}_{\geqslant 0} \right).
 \end{equation}

\item[(iii)] The Laurent expansion for $\psi (z)$ at $z=-k$ $\left(k \in \mathbb{Z}_{\geqslant 0}\right)$
      is given by
      \begin{equation}\label{LE-psi}
        \psi (z)= - \frac{1}{z+k}+\psi (k+1) + \sum_{n=2}^{\infty}\, \alpha_n\, (z+k)^{n-1},
      \end{equation}
   where
     \begin{equation}\label{LE-psi-a}
      \alpha_n = (-1)^n\, \zeta (n) + H_k^{(n)}.
     \end{equation}

\item[(iv)] The Laurent expansion for the polygamma function
$\psi^{(\ell)} (z)$ at $z=-k$ $\left(k \in \mathbb{Z}_{\geqslant 0}\right)$
      is given by
\begin{equation}\label{LE-polygamm}
 \psi^{(\ell)} (z) = \frac{(-1)^{\ell +1}\,\ell !}{(z+k)^{\ell +1}}
       + \sum_{n=\ell}^{\infty}\,\{n\}_{\ell}\,\alpha_{n+1}\,(z+k)^{n-\ell} \quad (\ell \in \mathbb{N}),
\end{equation}
where $\{\lambda \}_{\ell}$ $(\lambda \in \mathbb{C})$ is the falling factorial defined by
\begin{equation*}
  \{\lambda \}_{\ell}:=\left\{
                         \begin{array}{ll}
                           1 & (\ell=0) \\
                           \lambda (\lambda-1) \cdots (\lambda-\ell +1) & (\ell \in \mathbb{N}),
                         \end{array}
                       \right.
\end{equation*}
and $\alpha_n$ are given as in \eqref{LE-psi-a}.

\item[(v)]
\begin{equation}\label{psi-ext}
  \psi (z+m)= \psi (z) + \sum_{j=0}^{m-1}\, \frac{1}{z+j} \quad \left(m \in \mathbb{Z}_{\geqslant 0}\right)
\end{equation}
and
\begin{equation}\label{polygam-ext}
  \psi^{(n)} (z+m)= \psi^{(n)} (z) +(-1)^n\,n!\, \sum_{j=0}^{m-1}\, \frac{1}{(z+j)^{n+1}} \quad \left(m,\, n \in \mathbb{Z}_{\geqslant 0}\right).
\end{equation}

\end{itemize}
\end{lem}

\begin{proof}
  One may refer to \cite[pp. 4 and 24]{Sr-Ch-12} and \cite[Section 1.2]{Magn}.
Equation \eqref{LE-polygamm} can be derived by differentiating both sides of \eqref{LE-psi}
$\ell$-times.
\end{proof}

\vskip 3mm
\begin{lem}\label{LemB}
 Let $k \in \mathbb{Z}_{\geqslant 0}$. Then
\begin{equation}\label{LemB-eq1}
\lim_{z \rightarrow -k}\frac{\psi(z)}{\Gamma(z)}=(-1)^{k-1}\,k!;
\end{equation}

\begin{equation}\label{LemB-eq2}
\aligned
\lim_{z \rightarrow {-k}}\frac{\psi^2(z)-\psi'(z)}{\Gamma(z)}
  &= -\lim_{z \rightarrow {-k}}\, \frac{d}{dz} \left\{\frac{\psi(z)}{\Gamma(z)}\right\}\\
  &=2\,(-1)^{k-1}\,k!\,\psi(k+1);
\endaligned
\end{equation}

\begin{equation}\label{LemB-eq3}
\aligned
&\lim_{z \rightarrow {-k}}\frac{\psi^3(z)-3\,\psi(z)\,\psi'(z)+ \psi^{(2)}(z)}{\Gamma(z)}
 = \lim_{z \rightarrow {-k}}\,\frac{d^2}{dz^2} \left\{\frac{\psi(z)}{\Gamma(z)}\right\} \\
&\hskip 15mm =3\,(-1)^{k}\,k!\,\left\{\zeta (2) + H_k^{(2)}-\psi^2(k+1)\right\}.
\endaligned
\end{equation}

\end{lem}

\begin{proof}
 One finds from \eqref{Gamma-Res} and \eqref{psi-Res} that
\begin{equation*}
 \lim_{z \rightarrow -k}\frac{\psi(z)}{\Gamma(z)}= \lim_{z \rightarrow -k}\frac{(z+k)\, \psi(z)}{(z+k)\,\Gamma(z)}
   =(-1)^{k-1}\,k!,
\end{equation*}
which proves \eqref{LemB-eq1}.

\vskip 3mm
One may use \eqref{LE-psi} and \eqref{LE-polygamm} to obtain
  \begin{equation}\label{LemB-eq2-pf1}
\aligned
   \psi^2(z)-\psi'(z)& = - \frac{2\,\psi(k+1) }{z+k} +O(1)+O(z+k) \quad (z \rightarrow -k)\\
            &=- \frac{2\,\psi(k+1) }{z+k} +O(1)+o(1)\quad (z \rightarrow -k)\\
             &=- \frac{2\,\psi(k+1) }{z+k} +O(1)\quad (z \rightarrow -k).
\endaligned
  \end{equation}
Employing \eqref{Lem-A-eq1} and \eqref{LemB-eq2-pf1}, one can get
\begin{equation*}
\aligned
  &\lim_{z \rightarrow {-k}}\frac{\psi^2(z)-\psi'(z)}{\Gamma(z)}
    = \lim_{z \rightarrow {-k}} \,\frac{\sin (\pi z)\, \Gamma (1-z)}{\pi}\,\left\{\psi^2(z)-\psi'(z)\right\}\\
    &\hskip 3mm = \frac{\Gamma (1+k)}{\pi}\, \lim_{z \rightarrow {-k}}\,\frac{\sin (\pi z)}{z+k}\, (z+k) \left\{\psi^2(z)-\psi'(z)\right\}\\
     &\hskip 3mm = \frac{k!}{\pi}\,\pi\, \cos (\pi k)\, \left\{-2\,\psi(k+1)\right\},
\endaligned
\end{equation*}
which justifies \eqref{LemB-eq2}.

\vskip 3mm
\vskip 3mm
One may use \eqref{LE-psi} and \eqref{LE-polygamm} to derive
\begin{equation*}
  \psi^3(z)-3\,\psi(z)\,\psi'(z)+ \psi^{(2)}(z)
    = \frac{3 \left\{\alpha_2- \psi^2(k+1)\right\}}{z+k}+ O(1) \quad (z \rightarrow -k).
\end{equation*}
Now, a similar technique as in getting  \eqref{LemB-eq2} may verify \eqref{LemB-eq3}.
\end{proof}

\vskip 3mm
The next two theorems  are recalled (see \cite{1}).

\vskip 3mm

\begin{thm}\label{Cor2.1}
Let $x \in \mathbb{R} \setminus \mathbb{Z}_{\leqslant -1}$ and $m\in \mathbb{N}$. Then
\begin{equation}\label{e:15}
\sum_{k=1}^{\infty }\frac{(-1)^{k-1}}{k^{m}}\binom{x}{k}=\frac{%
(-1)^{m}}{m!}\frac{\partial^m}{\partial z^m}\frac{\Gamma (x+1)\Gamma (z)}{\Gamma (z+x)}\bigg|_{z=1}.
\end{equation}
\end{thm}

\begin{thm}\label{Thm25} Let $m,\, n\in\mathbb{N}$. Then
\begin{equation}\label{Thm25-eq}
\aligned
\sum_{k=0}^{\infty}\frac{(-1)^{n-1}}{(n+k+1)^{m+1}\binom{n+k}{k}}=&\sum_{k=1}^{n}\frac{(-1)^{k-1}}{k^m}\binom{n}{k}(H_n-H_{n-k})  \\
&-\frac{(-1)^m}{m!}\frac{\partial}{\partial x}\bigg\{\frac{\partial^{m}}{\partial{z}
^{m}}\frac{\Gamma(x+1)\Gamma (z)}{\Gamma (z+x)}\bigg\}\bigg|_{\begin{subarray}{l}
       x=n\\
       z=1
      \end{subarray}}.
\endaligned
\end{equation}
\end{thm}

\section{Main results}\label{sec-MR}

This section establishes our main findings.

\vskip 3mm

\begin{thm}\label{Thm2.1} Let  $x \in \mathbb{R} \setminus \mathbb{Z}_{\leqslant -1}$,
    $z \in \mathbb{R} \setminus \mathbb{Z}_{\leqslant 0}$,
  $m\in \mathbb{N}$, and  $n \in \mathbb{Z}_{\geqslant 0}$. Then
\begin{equation}\label{Thm2.1-eq}
  \aligned
&\sum_{k=0}^{\infty}\frac{H_k}{(n+k+1)^{m+1}\binom{n+k}{k}}\\
&\hskip 3mm =\frac{(-1)^n}{2}\sum_{k=1}^{n}\frac{(-1)^{k-1}}{k^m}\binom{n}{k}\left(H_{n-k}^2+H_{n-k}^{(2)}-H_n^{(2)}-H_n^2\right)\\
&\hskip 6mm-\frac{(-1)^{m+n}}{2\,m!}\frac{\partial^2}{\partial x^2}\left\{\frac{\partial^m}{\partial z^m}\frac{\Gamma(x+1)\Gamma (z)}{\Gamma (z+x)}\right\}\bigg|_{\begin{subarray}{l}
       x=n\\
       z=1
      \end{subarray}}\\
&\hskip 6mm+\frac{(-1)^{m+n}\,H_n}{m!}
\frac{\partial}{\partial x}\left\{\frac{\partial^{m}}{\partial{z}^{m}}\frac{\Gamma(x+1)\Gamma (z)}{\Gamma (z+x)}\right\}\bigg|_{\begin{subarray}{l}
       x=n\\
       z=1
      \end{subarray}}.
   \endaligned
\end{equation}
\end{thm}

\begin{proof}
Let $P(x)$ be the left and right members of \eqref{e:15}. Then
the use of \eqref{GBC} can write $P(x)$ as follows:
\begin{equation}\label{Thm2.1-eq-pf1}
 P(x):=\sum_{k=1}^{\infty }\frac{(-1)^{k-1}}{k^{m}\,k!}\,P_k(x) \quad {\rm and} \quad P_k(x):= \frac{\Gamma (x+1)}{\Gamma (x-k+1)}.
\end{equation}
Also
\begin{equation}\label{R-2-14}
P(x)=\frac{
(-1)^{m}}{m!}\frac{\partial^m}{\partial z^m}\frac{\Gamma (x+1)\Gamma (z)}{\Gamma (z+x)}\bigg|_{z=1}.
\end{equation}

 We first show that the series $P(x)$ in \eqref{Thm2.1-eq-pf1} can be differentiated
term-by-term for any point $x \in \mathbb{R} \setminus \mathbb{Z}_{\leqslant -1}$.
Note that
  \begin{equation*}
    P_k(x)= \frac{1}{\Gamma (x-k+1)}\cdot \Gamma (x+1),
  \end{equation*}
whose first factor is an entire function and the second factor is analytic
on $\mathbb{R} \setminus \mathbb{Z}_{\leqslant -1}$.
  Also, for each fixed $x \in  \mathbb{R} \setminus \mathbb{Z}_{\leqslant -1}$,
  \begin{equation*}
   \frac{d}{dx} P_k(x)=P'_k(x)= -\frac{ \Gamma (x+1)}{\Gamma (x-k+1)}\left\{\psi (x-k+1)-\psi (x+1)\right\}.
  \end{equation*}
By using asymptotic expansions for the ratio of gamma functions (see, e.g., \cite[p. 7]{Sr-Ch-12})
and the psi function (see, e.g., \cite[p. 36]{Sr-Ch-12}), one may obtain
\begin{equation*}
\aligned
  P'_k(x) &= O\left(x^k\right) \left\{\log (x-k+1)-\log(x+1)+ O\left(\frac{1}{x}\right)  \right\}\\
         &=O\left(x^k\right) \left\{ O\left(\frac{k}{x}\right) +O\left(\frac{1}{x}\right)\right\}\\
         &=O\left(x^{k-1}\right) \quad (|x| \rightarrow \infty).
\endaligned
\end{equation*}
 Thus, there exists $M>1$ so large that
     \begin{equation}\label{Thm2.1-eq-pf2}
       P'_k(x) = O\left(x^{k-1}\right)=O\left(x^{k}\right) \quad (|x|>M).
     \end{equation}
 Let $x_0$ be any point in $\mathbb{R} \setminus \mathbb{Z}_{\leqslant -1}$. One can choose $\delta>0$ so small that
  \begin{equation*}
   (x_0-\delta, x_0+\delta)\subseteq [x_0-\delta, x_0+\delta]\subseteq \mathbb{R} \setminus \mathbb{Z}_{\leqslant -1}.
  \end{equation*}
Now consider three cases: (i) $(x_0-\delta, x_0+\delta)\subseteq [x_0-\delta, x_0+\delta]\subseteq [-M,M]$.
  Since $P'_k(x)$ is continuous on the bounded closed interval $[x_0-\delta, x_0+\delta]$,
     $|P'_k(x)| \leqslant L_1$ for some $L_1>0$ and for all $x \in [x_0-\delta, x_0+\delta]$.
  Thus one may find
     \begin{equation*}
       \sum_{k=1}^{\infty }\left|\frac{(-1)^{k-1}}{k^{m}\,k!}\,P'_k(x)\right|\leqslant L_1 \sum_{k=1}^{\infty}\frac{1}{k!}=L_1(e-1)
     \end{equation*}
for all $x \in [x_0-\delta, x_0+\delta]$. In view of Weierstrass $M$-test, the series
$\sum\limits_{k=1}^{\infty }\frac{(-1)^{k-1}}{k^{m}\,k!}\,P'_k(x)$ converges uniformly on $(x_0-\delta, x_0+\delta)$.
Therefore the series $P(x)$ can be term-by-term differentiation at any point $x_0 \in (-M,M)$.

\vskip 3mm
(ii) $(x_0-\delta, x_0+\delta)\subseteq [x_0-\delta, x_0+\delta] \subseteq (-\infty,-M)\cup (M, \infty)$.
   Let $\eta:=\max\{|x_0-\delta|,|x_0+\delta|\}$. One may use \eqref{Thm2.1-eq-pf2} to see that
\begin{equation*}
   \sum_{k=1}^{\infty }\left|\frac{(-1)^{k-1}}{k^{m}\,k!}\,P'_k(x)\right|
    \leqslant L_2 \sum_{k=1}^{\infty}\, \frac{|x|^k}{k!} \leqslant L_2 \sum_{k=1}^{\infty}\, \frac{\eta^k}{k!}=L_2\,e^\eta
\end{equation*}
for some $L_2>0$ and for all $x \in [x_0-\delta, x_0+\delta]$. As in (ii), the series $P(x)$ can be term-by-term differentiation at any point $x_0 \in (-\infty,-M)\cup (M, \infty)$.

\vskip 3mm
(iii) $x_0=M$. Then either $(x_0-\delta, x_0+\delta) \subset [M-\delta, M+\delta]\subseteq (0,M)\cup (M,\infty)$ or
$(x_0-\delta, x_0+\delta) \subset [-M-\delta, -M+\delta]\subseteq  (-\infty,-M]\cup (-M,0)$. Proof of this case
leaves to the interested reader.

One therefore has
 \begin{equation}\label{dp(x)}
   P'(x) = \sum_{k=1}^{\infty }\frac{(-1)^{k-1}}{k^{m}\,k!}\,\frac{ \Gamma (x+1)}{\Gamma (x-k+1)}\left\{\psi (x+1)-\psi (x-k+1)\right\}.
 \end{equation}
Putting $x=n$ $(n \in \mathbb{N})$ in \eqref{dp(x)} and using \eqref{e:1}  affords
\begin{equation}\label{dp(x)-a}
\aligned
   P'(n) &= \sum_{k=1}^{\infty }\frac{(-1)^{k-1}}{k^{m}\,k!}\,\frac{ \Gamma (n+1)}{\Gamma (n-k+1)}\left\{\psi (n+1)-\psi (n-k+1)\right\}\\
        &=\sum_{k=1}^{n}\frac{(-1)^{k-1}}{k^{m}}\,\binom{n}{k}\,\left(H_n-\gamma\right)
       - \sum_{k=1}^{\infty }\frac{(-1)^{k-1}}{k^{m}}\,\binom{n}{k}\,\psi (n-k+1).
\endaligned
 \end{equation}
The following identity is known (see \cite[Example 3.7]{1}):
\begin{equation}\label{Thm2-1-pf3-cf}
\aligned
  \sum_{k=1}^{\infty}\frac{(-1)^{k-1}}{k^m}\binom{n}{k}\psi(n-k+1)
  &= \sum_{k=1}^{n}\frac{(-1)^{k-1}}{k^m}\binom{n}{k}\left(H_{n-k}-\gamma\right) \\
&  +(-1)^{n+1}\,\sum_{k=0}^{\infty}\frac{1}{(k+n+1)^{m+1}\,\binom{n+k}{k}}.
\endaligned
\end{equation}
Employing \eqref{Thm2-1-pf3-cf} in \eqref{dp(x)-a} yields
\begin{equation}\label{dp(x)-b}
\aligned
   P'(n) =& \sum_{k=1}^{n}\frac{(-1)^{k-1}}{k^{m}}\,\binom{n}{k}\,\left(H_n-H_{n-k}\right)\\
       & +(-1)^{n}\,\sum_{k=0}^{\infty}\frac{1}{(k+n+1)^{m+1}\,\binom{n+k}{k}},
\endaligned
 \end{equation}
which may be found to be equivalent to \eqref{Thm25-eq}.

\vskip 3mm
Term-by-term differentiation of $P'(x)$ in \eqref{dp(x)}, which can be confirmed using the preceding procedure, produces
\begin{equation*}
 \aligned
 P''(x) &= \sum_{k=1}^{\infty }\frac{(-1)^{k-1}}{k^{m}}\,\binom{x}{k}\,
        \big[\left\{\psi (x+1)-\psi (x-k+1)\right\}^2 \\
        &\hskip 35mm + \left\{\psi' (x+1)-\psi'(x-k+1)\right\}\big]\\
 &=\big\{\psi^2(x+1)+\psi^\prime(x+1)\big\}\sum_{k=1}^{\infty}\frac{(-1)^{k-1}}{k^m}\binom{x}{k}\\
&\hskip 3mm -2\psi(x+1)\sum_{k=1}^{\infty}\frac{(-1)^{k-1}}{k^m}\binom{x}{k}\psi(x-k+1)\\
&\hskip 3mm+\sum_{k=1}^{\infty}\frac{(-1)^{k-1}}{k^m}\binom{x}{k}\big\{\psi^2(x-k+1)-\psi^\prime(x-k+1)\big\}.
 \endaligned
\end{equation*}
Setting $x=n$ $(n \in \mathbb{N})$ and using \eqref{e:1} and \eqref{e:2} gives
\begin{equation}\label{Thm2.1-eq-pf3}
 \aligned
 P''(n)  &=\big\{(H_n-\gamma)^2+\zeta(2)-H_n^{(2)}\big\}\sum_{k=1}^{n}\frac{(-1)^{k-1}}{k^m}\binom{n}{k}\\
&\hskip 3mm -2\,(H_n-\gamma)\,\sum_{k=1}^{\infty}\frac{(-1)^{k-1}}{k^m}\binom{n}{k}\psi(n-k+1)\\
&\hskip 3mm+\sum_{k=1}^{\infty}\frac{(-1)^{k-1}}{k^m}\binom{n}{k}\big\{\psi^2(n-k+1)-\psi^\prime(n-k+1)\big\}.
 \endaligned
\end{equation}
Employing \eqref{e:1} and \eqref{e:2}, we obtain
\begin{equation}\label{Thm2-1-pf4}
\aligned
 & \sum_{k=1}^{\infty}\frac{(-1)^{k-1}}{k^m}\binom{n}{k}\big\{\psi^2(n-k+1)-\psi^\prime(n-k+1)\big\}\\
 &\hskip 3mm = \sum_{k=1}^{n}\frac{(-1)^{k-1}}{k^m}\binom{n}{k}\big\{(H_{n-k}-\gamma)^2 +H_{n-k}^{(2)}-\zeta (2)\big\}\\
  &\hskip 6mm + \sum_{k=n+1}^{\infty}\frac{(-1)^{k-1}}{k^m}\binom{n}{k}\big\{\psi^2(n-k+1)-\psi^\prime(n-k+1)\big\}.
\endaligned
\end{equation}
We get
\begin{equation*}
\aligned
  & \sum_{k=n+1}^{\infty}\frac{(-1)^{k-1}}{k^m}\binom{n}{k}\big\{\psi^2(n-k+1)-\psi^\prime(n-k+1)\big\}\\
&\hskip 3mm =\sum_{k=n+1}^{\infty}\frac{(-1)^{k-1}\,n!}{k^m\,k!}\,\frac{\psi^2(n-k+1)-\psi^\prime(n-k+1)}{\Gamma(n-k+1)},
\endaligned
\end{equation*}
which, upon setting $k-n-1=k'$ and dropping the prime on $k$, with the aid of \eqref{LemB-eq2}, offers
\begin{equation}\label{Thm2-1-pf4-a}
\aligned
  & \sum_{k=n+1}^{\infty}\frac{(-1)^{k-1}}{k^m}\binom{n}{k}\big\{\psi^2(n-k+1)-\psi^\prime(n-k+1)\big\}\\
&\hskip 3mm =\sum_{k=0}^{\infty}\frac{(-1)^{n+k}\,n!}{(k+n+1)^m\,(k+n+1)!}\,
\lim_{z \rightarrow -k}\frac{\psi^2(z)-\psi^\prime(z)}{\Gamma(z)}\\
&\hskip 3mm =2\,(-1)^{n+1}\,\sum_{k=0}^{\infty}\frac{n!}{(k+n+1)^m\,(k+n+1)!}\,
k!\,\psi(k+1)\\
&\hskip 3mm =2\,(-1)^{n+1}\,\sum_{k=0}^{\infty}\frac{\psi(k+1)}{(k+n+1)^{m+1}\,\binom{n+k}{k}}.
\endaligned
\end{equation}
Putting \eqref{Thm2-1-pf4-a} in \eqref{Thm2-1-pf4} provides
\begin{equation}\label{Thm2-1-pf4-b}
\aligned
 & \sum_{k=1}^{\infty}\frac{(-1)^{k-1}}{k^m}\binom{n}{k}\big\{\psi^2(n-k+1)-\psi^\prime(n-k+1)\big\}\\
 &\hskip 3mm = \sum_{k=1}^{n}\frac{(-1)^{k-1}}{k^m}\binom{n}{k}\big\{(H_{n-k}-\gamma)^2 +H_{n-k}^{(2)}-\zeta (2)\big\}\\
  &\hskip 6mm + 2\,(-1)^{n+1}\,\sum_{k=0}^{\infty}\frac{H_k -\gamma}{(k+n+1)^{m+1}\,\binom{n+k}{k}}.
\endaligned
\end{equation}

Setting \eqref{Thm2-1-pf3-cf} and \eqref{Thm2-1-pf4-b} in \eqref{Thm2.1-eq-pf3} gives
\begin{equation}\label{Thm2.1-eq-pf3-a}
 \aligned
 P''(n)  =& \big\{(H_n-\gamma)^2+\zeta(2)-H_n^{(2)}\big\}\sum_{k=1}^{n}\frac{(-1)^{k-1}}{k^m}\binom{n}{k}\\
& -2\,(H_n-\gamma)\,\sum_{k=1}^{n}\frac{(-1)^{k-1}}{k^m}\binom{n}{k}\left(H_{n-k}-\gamma\right) \\
&   +2\,(-1)^{n}\,H_n\,\sum_{k=0}^{\infty}\frac{1}{(k+n+1)^{m+1}\,\binom{n+k}{k}}\\
&+ \sum_{k=1}^{n}\frac{(-1)^{k-1}}{k^m}\binom{n}{k}\big\{(H_{n-k}-\gamma)^2 +H_{n-k}^{(2)}-\zeta (2)\big\}\\
  & + 2\,(-1)^{n+1}\,\sum_{k=0}^{\infty}\frac{H_k}{(k+n+1)^{m+1}\,\binom{n+k}{k}}.
 \endaligned
\end{equation}
From \eqref{dp(x)-b} and \eqref{Thm2.1-eq-pf3-a}, we derive
\begin{equation}\label{Th21-FR}
\aligned
  P''(n) -2\,H_{n}\,P'(n)= & \sum_{k=1}^{n}\frac{(-1)^{k-1}}{k^{m}}\,\binom{n}{k}\left(H_{n-k}^2 +H_{n-k}^{(2)}-H_n^2 -H_n^{(2)} \right)\\
                          & + 2\,(-1)^{n+1}\,\sum_{k=0}^{\infty}\frac{H_k}{(k+n+1)^{m+1}\,\binom{n+k}{k}}.
\endaligned
\end{equation}
Finally, \eqref{R-2-14} is used in the left member of \eqref{Th21-FR} to yield the desired result \eqref{Thm2.1-eq}.
\end{proof}

\vskip 3mm

The next corollary provides a proof of the Euler's classical formula \eqref{e:5}
as the particular case  of \eqref{Thm2.1-eq} when $n=0$.
\vskip 3mm

\begin{cor}\label{Cor2.2} Let $m \in \mathbb{Z}_{\geqslant 2}$. Then
\begin{equation*}
2\sum_{k=1}^{\infty}\frac{H_k}{(k+1)^m}=m\, \zeta(m+1)-\sum_{k=1}^{m-2}\,\zeta(k+1)\,\zeta(m-k).
\end{equation*}
\end{cor}
\begin{proof}Setting $n=0$ in \eqref{Thm2.1-eq} may yield
\begin{equation}\label{e:20}
\sum_{k=0}^{\infty}\frac{H_k}{(k+1)^{m+1}}
=\frac{(-1)^{m+1}}{2m!}\frac{\partial^2}{\partial x^2}\left\{\frac{\partial^m}{\partial z^m}\frac{\Gamma(x+1)\Gamma (z)}{\Gamma (z+x)}\right\}\bigg|_{\begin{subarray}{l}
       x=0\\
       z=1
      \end{subarray}}.
\end{equation}
Interchanging the order of differentiations with respect to $x$ and $z$ in (\ref{e:20}),
which can be guaranteed since
\begin{equation*}
  \frac{\partial^2}{\partial x^2}\left\{\frac{\partial^m}{\partial z^m}\frac{\Gamma(x+1)\Gamma (z)}{\Gamma (z+x)}\right\}
\quad \text{and} \quad
\frac{\partial^m}{\partial z^m} \left\{\frac{\partial^2}{\partial x^2}\frac{\Gamma(x+1)\Gamma (z)}{\Gamma (z+x)}\right\}
\end{equation*}
are analytic and so continuous on $x \in \mathbb{C}\setminus \mathbb{Z}_{\leqslant -1}$ and $z \in \mathbb{C}\setminus \mathbb{Z}_{\leqslant 0}$,
 replacing $m$ by $m-1$, and noticing that
\begin{align}\label{e:21}
\frac{d^2}{dx^2}\frac{\Gamma(x+1)\Gamma (z)}{\Gamma (z+x)}\bigg|_{x=0}=\psi^2(z)+2 \gamma  \psi(z)-\psi ^\prime(z)+\gamma ^2+\frac{\pi ^2}{6},
\end{align}
we obtain
\allowdisplaybreaks
\begin{align*}
&2\sum_{k=1}^{\infty}\frac{H_k}{(k+1)^{m}}=\frac{(-1)^{m}}{(m-1)!}\frac{\partial^{m-1}}{\partial z^{m-1}}\left[\psi^2(z)+2 \gamma  \psi(z)-\psi ^\prime(z)\right]_{z=1}\\
&=\frac{(-1)^{m}}{(m-1)!}\bigg\{\sum_{k=0}^{m-1}\binom{m-1}{k}\psi^{(k)}(1)\psi^{(m-k-1)}(1)+2\gamma\psi^{(m-1)}(1)-\psi^{(m)}(1)\bigg\}\\
&=\frac{(-1)^{m}}{(m-1)!}\bigg\{\sum_{k=1}^{m-2}\binom{m-1}{k}\psi^{(k)}(1)\psi^{(m-k-1)}(1)-\psi^{(m)}(1)\\
&\hskip 3mm +\underbrace{2\psi(1)\psi^{({(m-1)})}(1)+2\gamma\psi^{(m-1)}(1)}_{0}\bigg\},
\end{align*}
which, upon using \eqref{e:2}, yields the desired result.
\end{proof}

\vskip 3mm

\begin{thm}\label{Theorem2.3}  Let $n \in \mathbb{Z}_{\geqslant 0}$  and $m\in\mathbb{N}$. Then
\begin{equation}\label{Thm23-eq}
\aligned
&\sum_{k=0}^{\infty}\frac{H_k^2-H_k^{(2)}}{(n+k+1)^{m+1}\binom{n+k}{k}}
  =\frac{(-1)^{n-1}}{3}\sum_{k=1}^{n}\frac{(-1)^{k-1}}{k^m}\binom{n}{k}\\
 & \times \Big\{H_n^3+2H_n^{(3)} +3H_nH_n^{(2)}-H_{n-k}^3-2H_{n-k}^{(3)}-3H_{n-k}H_{n-k}^{(2)}\Big\}\\
&+\frac{(-1)^{m+n}}{m!}\Big\{ \big(H_n^2+H_n^{(2)}\big)\,F_1(n,m)- H_n\,F_2(n,m)+ \frac{1}{3}\,F_3(n,m)\Big\},
\endaligned
\end{equation}
where
\begin{equation*}
  F_1(n,m):= \frac{\partial}{\partial x}\bigg\{\frac{\partial^m}{\partial z^m}\frac{\Gamma(x+1)\Gamma(z)}{\Gamma(x+z)}\bigg\}\bigg|_{\begin{subarray}{l}
       x=n\\
       z=1
      \end{subarray}},
\end{equation*}
\begin{equation*}
  F_2(n,m):= \frac{\partial^2}{\partial x^2}\bigg\{\frac{\partial^m}{\partial z^m}\frac{\Gamma(x+1)\Gamma(z)}{\Gamma(x+z)}\bigg\}\bigg|_{\begin{subarray}{l}
       x=n\\
       z=1
      \end{subarray}},
\end{equation*}
and
\begin{equation*}
  F_3(n,m):= \frac{\partial^3}{\partial x^3}\bigg\{\frac{\partial^m}{\partial z^m}\frac{\Gamma(x+1)\Gamma(z)}{\Gamma(x+z)}\bigg\}\bigg|_{\begin{subarray}{l}
       x=n\\
       z=1
      \end{subarray}};
\end{equation*}
\end{thm}

\begin{proof}
As in the proof of Theorem \ref{Thm2.1}, let $P(x)$ be the same as in \eqref{Thm2.1-eq-pf1}. Then we may differentiate
$P(x)$ with respect to $x$ three times term-by-term. Then putting $x=n$ $(n \in \mathbb{N})$ in the $P^{(3)}(x)$, with the aid of \eqref{e:1} and \eqref{e:2},  we find
\begin{equation}\label{Th33-pf1}
\aligned
&P^{(3)}(n)=\big\{(H_n-\gamma)^3+3(H_n-\gamma)(\zeta(2)-H_n^{(2)})\\
&\hskip 20mm +2(H_n^{(3)}-\zeta(3))\big\}\sum_{k=1}^{n}\frac{(-1)^{k-1}}{k^m}\binom{n}{k}\\
&\hskip 3mm-3\left\{(H_n-\gamma)^2+(\zeta(2)-H_n^{(2)})\right\}\sum_{k=1}^{\infty}\frac{(-1)^{k-1}}{k^m}\binom{n}{k}\psi(n-k+1)\\
&\hskip 3mm +3(H_n-\gamma)\sum_{k=1}^{\infty}\frac{(-1)^{k-1}}{k^m}\binom{n}{k}\left\{\psi^2(n-k+1)-\psi^\prime(n-k+1)\right\}\\
&\hskip 3mm -\sum_{k=1}^{\infty}\frac{(-1)^{k-1}}{k^m}\binom{n}{k}\big\{\psi^3(n-k+1)-3\psi(n-k+1)\psi^\prime(n-k+1)\\
&\hskip 45mm +\psi^{\prime\prime}(n-k+1)\big\}.
\endaligned
\end{equation}
Here, we consider
\begin{equation*}
  \aligned
&\sum_{k=1}^{\infty}\frac{(-1)^{k-1}}{k^m}\binom{n}{k}\big\{\psi^3(n-k+1)-3\psi(n-k+1)\psi^\prime(n-k+1)\\
&\hskip 35mm +\psi^{\prime\prime}(n-k+1)\big\}\\
&=\sum_{k=1}^{n}\frac{(-1)^{k-1}}{k^m}\binom{n}{k}\big\{(H_{n-k}-\gamma)^3-3(H_{n-k}-\gamma)(\zeta (2)-H_{n-k}^{(2)})\\
&\hskip 35mm +2\,H_{n-k}^{(3)}-2\, \zeta (3)\big\}\\
 &+ \sum_{k=n+1}^{\infty}\frac{(-1)^{k-1}}{k^m}\binom{n}{k}\big\{\psi^3(n-k+1)-3\psi(n-k+1)\psi^\prime(n-k+1)\\
&\hskip 35mm +\psi^{\prime\prime}(n-k+1)\big\}.
 \endaligned
\end{equation*}
As in getting \eqref{Thm2-1-pf4-b}, we use \eqref{LemB-eq3} to obtain
\begin{equation*}
\aligned
  & \sum_{k=n+1}^{\infty}\frac{(-1)^{k-1}}{k^m}\binom{n}{k}\big\{\psi^3(n-k+1)-3\psi(n-k+1)\psi^\prime(n-k+1)\\
&\hskip 35mm +\psi^{\prime\prime}(n-k+1)\big\}\\
&= 3\,(-1)^n\, \sum_{k=0}^{\infty}\, \frac{H_k^{(2)}-H_k^2 + \zeta (2) +2\gamma\,H_k-\gamma^2}{(k+n+1)^{m+1}\,\binom{k+n}{k}}
  \endaligned
\end{equation*}
We therefore have
\begin{equation}\label{Th33-pf2}
\aligned
  &\sum_{k=1}^{\infty}\frac{(-1)^{k-1}}{k^m}\binom{n}{k}\big\{\psi^3(n-k+1)-3\psi(n-k+1)\psi^\prime(n-k+1)\\
&\hskip 35mm +\psi^{\prime\prime}(n-k+1)\big\}\\
&=\sum_{k=1}^{n}\frac{(-1)^{k-1}}{k^m}\binom{n}{k}\big\{(H_{n-k}-\gamma)^3-3(H_{n-k}-\gamma)(\zeta (2)-H_{n-k}^{(2)})\\
&\hskip 35mm +2\,H_{n-k}^{(3)}-2\, \zeta (3)\big\}\\
 &\hskip 3mm +3\,(-1)^n\, \sum_{k=0}^{\infty}\, \frac{H_k^{(2)} + \zeta (2) -(H_k-\gamma)^2}{(k+n+1)^{m+1}\,\binom{k+n}{k}}.
\endaligned
\end{equation}

Employing \eqref{Thm2-1-pf3-cf}, \eqref{Thm2-1-pf4-b} and \eqref{Th33-pf2} in \eqref{Th33-pf1}, we may find
\begin{equation*}
\aligned
&P^{(3)}(n)=\big\{(H_n-\gamma)^3+3(H_n-\gamma)(\zeta(2)-H_n^{(2)})\\
&\hskip 20mm +2(H_n^{(3)}-\zeta(3))\big\}\sum_{k=1}^{n}\frac{(-1)^{k-1}}{k^m}\binom{n}{k}\\
\endaligned
\end{equation*}
\begin{equation*}
\aligned
&-3\left\{(H_n-\gamma)^2+(\zeta(2)-H_n^{(2)})\right\}\bigg[\sum_{k=1}^{n}\frac{(-1)^{k-1}}{k^m}\binom{n}{k}\left(H_{n-k}-\gamma\right) \\
&\hskip 30mm   +(-1)^{n+1}\,\sum_{k=0}^{\infty}\frac{1}{(k+n+1)^{m+1}\,\binom{n+k}{k}}\bigg]\\
\endaligned
\end{equation*}

\begin{equation*}
\aligned
&\hskip 3mm +3(H_n-\gamma)\bigg[\sum_{k=1}^{n}\frac{(-1)^{k-1}}{k^m}\binom{n}{k}\big\{(H_{n-k}-\gamma)^2 +H_{n-k}^{(2)}-\zeta (2)\big\}\\
  &\hskip 30mm + 2\,(-1)^{n+1}\,\sum_{k=0}^{\infty}\frac{H_k -\gamma}{(k+n+1)^{m+1}\,\binom{n+k}{k}}\bigg]\\
\endaligned
\end{equation*}

\begin{equation*}
\aligned
&\hskip 3mm -\sum_{k=1}^{n}\frac{(-1)^{k-1}}{k^m}\binom{n}{k}\big\{(H_{n-k}-\gamma)^3-3(H_{n-k}-\gamma)(\zeta (2)-H_{n-k}^{(2)})\\
&\hskip 35mm +2\,H_{n-k}^{(3)}-2\, \zeta (3)\big\}\\
 &\hskip 3mm -3\,(-1)^n\, \sum_{k=0}^{\infty}\, \frac{H_k^{(2)}-H_k^2 + \zeta (2) +2\gamma\,H_k-\gamma^2}{(k+n+1)^{m+1}\,\binom{k+n}{k}}.
\endaligned
\end{equation*}

Finally, using \eqref{dp(x)-b}, \eqref{Thm2.1-eq-pf3}, and the  expression $P^{(3)}(n)$ just obtained,
 as in getting the result in Theorem \ref{Thm2.1}, we can readily establish \eqref{Thm23-eq}.
\end{proof}

\vskip 3mm
A particular case of \eqref{Thm23-eq} when $n=0$ produces the identity in Corollary \ref{Cor2.4}.
\vskip 3mm

\begin{cor}\label{Cor2.4}
Let  $m\in\mathbb{N}$. Then
\begin{equation}\label{Cor2.4-eq}
\aligned
&\sum_{k=0}^{\infty}\frac{H_k^2-H_k^{(2)}}{(k+1)^{m+1}} = \frac{(-1)^{m}}{3\, m!}\,F_3(0,m)\\
&\hskip 3mm =  \frac{(m+1)(m+2)}{3}\, \zeta (m+3)
 - \sum_{j=1}^{m-1}\,(j+1)\,\zeta (j+2)\,\zeta (m+1-j)\\
&\hskip 6mm +\frac{1}{m}\, \sum_{\ell=1}^{m-1}\,(m-\ell)\,\zeta (m-\ell+1)\,\sum_{j=1}^{\ell-1}\,\zeta (j+1)\,\zeta (\ell-j+1).
\endaligned
\end{equation}
Also
\begin{equation}\label{Cor2.4-eq-a}
\aligned
&\sum_{k=1}^{\infty}\frac{H_k^2-H_k^{(2)}}{k^{m+1}}=\frac{(-1)^{m}}{3\, m!}\,F_3(0,m)\\
&\hskip 6mm  + (m+2)\,\zeta (m+3) - \sum_{j=1}^{m}\,\zeta (j+1)\,\zeta (m+2-j) \\
&\hskip 3mm =  \frac{(m+2)(m+4)}{3}\, \zeta (m+3)-\zeta (2)\,\zeta (m+1) \\
&\hskip 6mm - 2\sum_{j=2}^{m}\,\zeta (j+1)\,\zeta (m+2-j)
 - \sum_{j=1}^{m-1}\,j\,\zeta (j+2)\,\zeta (m+1-j)\\
&\hskip 6mm  +\frac{1}{m}\, \sum_{\ell=1}^{m-1}\,(m-\ell)\,\zeta (m-\ell+1)\,\sum_{j=1}^{\ell-1}\,\zeta (j+1)\,\zeta (\ell-j+1),
\endaligned
\end{equation}
where
\begin{equation*}
  F_3(0,m)= \frac{\partial^3}{\partial x^3}\bigg\{\frac{\partial^m}{\partial z^m}\frac{\Gamma(x+1)\Gamma(z)}{\Gamma(x+z)}\bigg\}\bigg|_{\begin{subarray}{l}
       x=0\\
       z=1
      \end{subarray}}.
\end{equation*}
\end{cor}

\begin{proof}
  Setting $n=0$ \eqref{Thm23-eq} gives
\begin{equation}\label{Cor24-pf1}
\aligned
&\sum_{k=0}^{\infty}\frac{H_k^2-H_k^{(2)}}{(k+1)^{m+1}}
  =
\frac{(-1)^{m}}{3\, m!}\,  F_3(0,m),
\endaligned
\end{equation}
where
\begin{equation*}
\aligned
  F_3(0,m) &= \frac{d^{m-1}}{d z^{m-1}}\bigg\{\frac{d}{dz}\frac{\partial^3}{\partial x^3}\frac{\Gamma(x+1)\Gamma(z)}{\Gamma(x+z)}\bigg|_{x=0}\bigg\} \bigg|_{z=1}.
\endaligned
\end{equation*}
We obtain
\begin{equation}\label{Cor24-pf2-a}
\aligned
&\frac{d}{d z}\bigg\{\frac{\partial^3}{\partial x^3}\frac{\Gamma(x+1)\Gamma(z)}{\Gamma(x+z)}\bigg|_{x=0}\bigg\}
  = -3\,\left(\gamma^2+ \zeta (2)\right)\,\psi'(z)+3 \gamma\,\psi^{(2)}(z) \\
 &  -\psi^{(3)}(z) + 3\,\psi(z)\,\psi^{(2)}(z)+3\,\left(\psi'(z)\right)^2
  -6\gamma\,\psi(z)\,\psi'(z)-3 (\psi(z))^2\,\psi'(z).
\endaligned
\end{equation}
Note that
     \begin{equation}\label{psi-rd-1-a}
   \frac{d^\ell}{dz^\ell}f(z):=\frac{d^\ell}{dz^\ell}(\psi(z))^2=\sum_{j=0}^{\ell}\,\binom{\ell}{j}\,\psi^{(j)}(z)\,\psi^{(\ell-j)}(z)
       \quad \left(\ell \in \mathbb{Z}_{\geqslant 0}\right),
     \end{equation}
which, upon putting $z=1$ and using \eqref{e:1} and \eqref{e:2}, yields
 \begin{equation}\label{psi-rd-1}
  \aligned
      &\frac{d^\ell}{dz^\ell}(\psi(z))^2\Big|_{z=1}  =\sum_{j=0}^{\ell}\,\binom{\ell}{j}\,\psi^{(j)}(1)\,\psi^{(\ell-j)}(1)\\
           &= 2 \gamma \,(-1)^\ell\,\ell!\,\zeta (\ell+1) +(-1)^\ell\,\ell!\, \sum_{j=1}^{\ell-1}\, \zeta (j+1)\, \zeta (\ell-j+1).
  \endaligned
           \end{equation}

Employing \label{Cor24-pf2-a}, with the aid of \eqref{e:1}, \eqref{e:2} and \eqref{psi-rd-1-a}, we derive
\begin{equation*}
\aligned
&\frac{\partial^m}{\partial z^m}\bigg\{\frac{\partial^3}{\partial x^3}\frac{\Gamma(x+1)\Gamma(z)}{\Gamma(x+z)}\bigg|_{x=0}\bigg\}
  = -3\,\left(\gamma^2+ \zeta (2)\right)\,\psi^{(m)}(z)+3 \gamma\,\psi^{(m+1)}(z) \\
 &\hskip 10mm   -\psi^{(m+2)}(z) +3\, \sum_{j=0}^{m-1}\,\binom{m-1}{j}\,\psi^{(2+j)}(z)\, \psi^{(m-1-j)}(z)\\
&\hskip 10mm +3\,\sum_{j=0}^{m-1}\,\binom{m-1}{j}\,\psi^{(j+1)}(z)\, \psi^{(m-j)}(z)\\
& \hskip 10mm -6\gamma\,\sum_{j=0}^{m-1}\,\binom{m-1}{j}\,\psi^{(j)}(z)\, \psi^{(m-j)}(z) \\
 &\hskip 10mm -3\, \sum_{\ell=0}^{m-1}\,\binom{m-1}{\ell}\,f^{(\ell)}(z)\, \psi^{(m-\ell)}(z),
\endaligned
\end{equation*}
  which, upon setting $z=1$, yields
 \begin{equation*}
\aligned
&\frac{(-1)^m}{3\,m!}\,F_3(0,m)
  = \zeta (2)\,\zeta (m+1)
   + \frac{(m+1)(m+2)}{3}\, \zeta (m+3) \\
&-\frac{1}{m} \sum_{j=0}^{m-2}\,(j+1)(j+2)\,\zeta (3+j)\,\zeta (m-j) \\
& - \frac{1}{m}\,\sum_{j=0}^{m-1}\,(j+1)\,(m-j)\,\zeta (j+2)\,\zeta (m+1-j)\\
&+\frac{1}{m}\, \sum_{\ell=1}^{m-1}\,(m-\ell)\,\zeta (m-\ell+1)\,\sum_{j=1}^{\ell-1}\,\zeta (j+1)\,\zeta (\ell-j+1).
\endaligned
\end{equation*}
Finally, the last expression may be simplified to yield the desired result \eqref{Cor2.4-eq}.

\vskip 3mm
Using \eqref{e:5-a}, we may obtain
\begin{equation}\label{Cor34-EulerI}
\aligned
& \sum_{k=0}^{\infty}\frac{H_k^2-H_k^{(2)}}{(k+1)^{m+1}}
 = \sum_{k=1}^{\infty}\frac{H_k^2-H_k^{(2)}}{k^{m+1}}-  (m+2)\,\zeta (m+3)\\
   &\hskip 23mm  + \sum_{k=1}^{m}\,\zeta (k+1)\,\zeta (m+2-k) \quad (m \in \mathbb{N}).
\endaligned
\end{equation}
Employing \eqref{Cor34-EulerI} in \eqref{Cor2.4-eq} produces \eqref{Cor2.4-eq-a}.
\end{proof}

\vskip 3mm
\begin{thm}\label{Theorem2.5}
Let $p \in \mathbb{C} \setminus \mathbb{Z}_{\leqslant 0}$,  $x \in \mathbb{C} \setminus \mathbb{Z}_{\leqslant -1}$ and  $m \in \mathbb{N}$. Then
\begin{equation}\label{e:25}
\sum_{k=0}^{\infty}\frac{(-1)^k}{(p+k)^{m+1}}\binom{x}{k}=\frac{(-1)^{m}}{m!}\frac{\partial^{m}}{\partial s^{m}}\frac{\Gamma(x+1)\Gamma(s)}{\Gamma(x+s+1)}\bigg|_{s=p}.
\end{equation}
\end{thm}

\begin{proof}
Using \eqref{GF}, one may find
\begin{equation}\label{Th25-pf-1}
  \frac{1}{t^{m+1}}=\frac{1}{m!}\int_{0}^{\infty}u^{m}e^{-t\,u}\,du \quad \left(m \in \mathbb{Z}_{\geqslant 0},\,\, \Re(t)>0\right).
\end{equation}
Employing \eqref{Th25-pf-1}, one can obtain
\begin{align*}
\sum_{k=0}^{\infty}\frac{(-1)^k}{(p+k)^{m+1}}\binom{x}{k}=\frac{1}{m!}\sum_{k=0}^{\infty}(-1)^k\binom{x}{k}\int_{0}^{\infty}u^{m}e^{-(p+k)u}\,du.
\end{align*}
Here and in the following,  $\Re(p)>0$ is assumed.
Interchanging the order of integration and summation,  one may get
\begin{equation}\label{Th25-pf-2}
\aligned
\sum_{k=0}^{\infty}\frac{(-1)^k}{(p+k)^{m+1}}\binom{x}{k}&=\frac{1}{m!}\int_{0}^{\infty}\,u^{m}e^{-pu}\sum_{k=0}^{\infty}(-1)^k\binom{x}{k}e^{-ku}du\\
&=\frac{1}{m!}\int_{0}^{\infty}u^{m}\,e^{-pu}(1-e^{-u})^x\,du.
\endaligned
\end{equation}
In order to verify the above term-by-term integration, let
  \begin{equation*}
    g_k(u):= u^{m}\,e^{-pu}\, (-1)^k\,\binom{x}{k}\,e^{-ku} \quad \left(u \in \mathbb{R}_{>0},\, k \in \mathbb{Z}_{\geqslant 0}\right).
  \end{equation*}
Note that
  \begin{equation*}
    \left|\binom{x}{k}\right|\leqslant \frac{|x|(|x|+1)\cdots (|x|+k-1)}{k!}= \frac{1}{\Gamma (|x|)}\frac{\Gamma (k+|x|)}{\Gamma (k+1)}.
  \end{equation*}
Holding $x$ fixed and using the asymptotic expansion of ratio of gamma functions (see, e.g., \cite[p. 7]{Sr-Ch-12}),
   we may find
\begin{equation*}
 \left|\binom{x}{k}\right| = O \left(k^{|x|-1}\right) \quad (k \rightarrow \infty).
\end{equation*}
That is, there exist $M \in \mathbb{R}_{>0}$ and $N \in \mathbb{Z}_{\geqslant 2}$ such that
  \begin{equation*}
    \left|\binom{x}{k}\right| \leqslant M\, k^{|x|-1} \,\,\,\text{ for all} \,\,\, k \in \mathbb{Z}_{\geqslant N},
  \end{equation*}
in particular,
\begin{equation*}
    \left|\binom{x}{k}\right| \leqslant M \,\,\,\text{ for all} \,\,\, k \in \mathbb{Z}_{\geqslant N} \,\,\,\text{and} \,\,\, |x|< 1.
  \end{equation*}
Thus we may find that, for $|x|< 1$ and  $\Re(p)>0$,
\begin{equation*}
\aligned
  \sum_{k=N}^{\infty}\, \left|g_k(u) \right| & \leqslant M\,u^m\,e^{-\Re(p)\,u} \sum_{k=N}^{\infty}\, e^{-k u}\\
       & =  M\,u^m\,e^{-\Re(p)\,u}\, \frac{e^{-N u}}{1-e^{-u}}\\
       & =  M\,u^m\,e^{-\Re(p)\,u}\, \frac{e^{-(N-1) u}}{e^u-1}.
\endaligned
\end{equation*}
Since $e^u -1 \geqslant u$ for all $u \in \mathbb{R}_{>0}$, we obtain
\begin{equation*}
   \sum_{k=N}^{\infty}\, \left|g_k(u) \right|  \leqslant M\,u^{m-1}\, e^{-(\Re(p) +N-1)u}
\end{equation*}
and, with the aid of \eqref{Th25-pf-1},
\begin{equation*}
\aligned
   \int_{0}^{\infty}\,\sum_{k=N}^{\infty}\, \left|g_k(u) \right| \,du &\leqslant M\, \int_{0}^{\infty}\, u^{m-1}\, e^{-(\Re(p) +N-1)u}\,du\\
       &= \frac{M\, (m-1)!}{(\Re(p) +N-1)^m} <\infty.
\endaligned
\end{equation*}
Now, by employing the Lebesgue dominated convergence theorem (see, e.g., \cite[p. 53]{Foll}),
\eqref{Th25-pf-2} may be justified.

Making the change of variable $1-e^{-u}=y$, we derive that
\begin{align*}
\sum_{k=0}^{\infty}\frac{(-1)^k}{(p+k)^{m+1}}\binom{x}{k}&=\frac{(-1)^{m}}{m!}\int_{0}^{1}y^x(1-y)^{p-1}\log^{m}(1-y)\,dy\\
&=\frac{(-1)^{m}}{m!}\int_{0}^{1}y^x\frac{\partial^{m}}{\partial s^{m}}(1-y)^{s-1}\big|_{s=p}dy\\
&=\frac{(-1)^{m}}{m!}\frac{\partial^{m}}{\partial s^{m}}\,B(x+1,s)\big|_{s=p},
\end{align*}
where $B(x+1,s)$ is the Beta function in \eqref{beta}.

We thus proved the following identity:
\begin{equation}\label{Th25-pf-3}
\aligned
  & \sum_{k=0}^{\infty}\frac{(-1)^k}{(p+k)^{m+1}}\binom{x}{k}= \frac{(-1)^{m}}{m!}\frac{\partial^{m}}{\partial s^{m}}\frac{\Gamma(x+1)\Gamma(s)}{\Gamma(x+s+1)}\bigg|_{s=p}\\
  & \hskip 20mm  \left(\Re(p)>0,\, \, |x|<1,\,\, m \in \mathbb{N}\right).
\endaligned
\end{equation}
One can observe that both sides of \eqref{Th25-pf-3} are analytic functions of both variables $p$ and $x$
in the wider domains $p \in \mathbb{C} \setminus \mathbb{Z}_{\leqslant 0}$ and  $x \in \mathbb{C} \setminus \mathbb{Z}_{\leqslant -1}$.
Finally, by the principle of analytic continuation, the desired identity \eqref{e:25} can hold true for the given domains.
\end{proof}

\vskip 3mm
\begin{cor}\label{Cor2.6} Let $p \in \mathbb{C} \setminus \mathbb{Z}_{\leqslant 0}$ and $m \in \mathbb{Z}_{\geqslant 0}$. Then
\begin{equation}\label{Cor26-eq}
\aligned
&\sum_{k=0}^{\infty}\frac{\left(H_k-2H_{2k}\right)\binom{2k}{k}}{4^k(p+k)^{m+1}}\\
&\hskip 5mm =\frac{\sqrt{\pi}\,(-1)^{m}}{m!}\frac{d^{m}}{d s^{m}}\frac{\Gamma(s)}{\Gamma(s+\tfrac{1}{2})}\left\{\psi\left(\tfrac{1}{2}\right)-\psi\left(s+\tfrac{1}{2}\right)\right\}\Big|_{s=p}.
\endaligned
\end{equation}
\end{cor}

\begin{proof}
We may get  \eqref{Cor26-eq} by differentiating both sides of \eqref{e:25}, with respect to $x$, and then setting $x=-\frac{1}{2}$
in the resultant identity, and using several formulas for gamma and $\psi$ functions such as
\begin{equation}\label{Cor36-eq-e}
 \Gamma\left(\tfrac{1}{2}-k \right)=\sqrt{\pi}\, (-1)^k\, \frac{2^{2k}\,k!}{(2k)!} \quad \left(k \in \mathbb{Z}_{\geqslant 0}\right),
\end{equation}

\begin{equation}\label{Cor36-eq-d}
  \psi\left(\tfrac{1}{2}\right)-\psi\left(\tfrac{1}{2}-k \right)= H_k - 2\,H_{2k} \quad \left(k \in \mathbb{Z}_{\geqslant 0}\right)
\end{equation}
and
\begin{equation}\label{Cor36-eq-e}
  \psi\left(\tfrac{1}{2}-k\right)=\psi\left(\tfrac{1}{2}+k \right) \quad \left(k \in \mathbb{Z}_{\geqslant 0}\right).
\end{equation}

\end{proof}


\vskip 3mm
\begin{thm}\label{Theorem2.7}
Let $p \in \mathbb{C} \setminus \mathbb{Z}_{\leqslant 0}$,  $m \in \mathbb{Z}_{\geqslant 0}$, and $n \in \mathbb{Z}_{\geqslant 0}$.
Then
\begin{equation}\label{Thm27-eq}
\aligned
\sum_{k=1}^{\infty}\frac{(-1)^n}{k(p+n+k)^{m+1}\binom{n+k}{k}}&=\sum_{k=0}^{n}\frac{(-1)^{k}}{(p+k)^{m+1}}\binom{n}{k}\big(H_n-H_{n-k}\big)\\
&-\frac{(-1)^{m}}{m!}\frac{\partial}{\partial x}\frac{\partial^{m}}{\partial s^{m}}\frac{\Gamma(x+1)\Gamma(s)}{\Gamma(x+s+1)}\bigg|_{\begin{subarray}{l}
       x=n\\
       s=p
      \end{subarray}}.
\endaligned
\end{equation}
\end{thm}

\begin{proof}
Denote both sides of \eqref{e:25} by $Q(x)$. Then, differentiating the left member of  \eqref{e:25}, with respect to $x$, gives
\begin{equation*}
Q^\prime(x)=\sum_{k=0}^{\infty}\frac{(-1)^{k}}{(p+k)^{m+1}}\binom{x}{k}(\psi(x+1)-\psi(x-k+1)),
\end{equation*}
which, upon setting $x=n$,  yields
\begin{align}\label{e:26}
Q^\prime(n)&=\psi(n+1)\sum_{k=0}^{n}\frac{(-1)^{k}}{(p+k)^{m+1}}\binom{n}{k}-\sum_{k=0}^{\infty}\frac{(-1)^{k}}{(p+k)^{m+1}}\binom{n}{k}\psi(n-k+1).
\end{align}
Split the second sum in \eqref{e:26} into two parts as follows:
\begin{align*}
\sum_{k=0}^{\infty}\frac{(-1)^{k}}{(p+k)^{m+1}}\binom{n}{k}\psi(n-k+1)&=\sum_{k=0}^{n}\frac{(-1)^{k}}{(p+k)^{m+1}}\binom{n}{k}(H_{n-k}-\gamma)\notag\\
&+\sum_{k=n+1}^{\infty}\frac{(-1)^{k}}{(p+k)^{m+1}}\binom{n}{k}\psi(n-k+1).
\end{align*}
Letting  $k-n-1=k^\prime$ and then dropping the prime on $k$, and using \eqref{LemB-eq1}, we get
\begin{align*}
&\sum_{k=n+1}^{\infty}\frac{(-1)^{k}}{(p+k)^{m+1}}\binom{n}{k}\psi(n-k+1)\notag\\
&\hskip 5mm=(-1)^{n+1}\,\sum_{k=0}^{\infty}\frac{(-1)^{k}}{(p+n+k+1)^{m+1}}\frac{n!}{(n+k+1)!}\frac{\psi(-k)}{\Gamma(-k)}\\
&\hskip 5mm=\sum_{k=0}^{\infty}\frac{(-1)^{n}}{(p+n+k+1)^{m+1}(n+k+1)\binom{n+k}{k}}.
\end{align*}
Thus,
\begin{equation}\label{e:27}
\aligned
&\sum_{k=0}^{\infty}\frac{(-1)^{k}}{(p+k)^{m+1}}\binom{n}{k}\psi(n-k+1)
     =\sum_{k=0}^{n}\frac{(-1)^{k}}{(p+k)^{m+1}}\binom{n}{k}(H_{n-k}-\gamma)\\
&\hskip 45mm +\sum_{k=0}^{\infty}\frac{(-1)^{n}}{(p+n+k+1)^{m+1}(n+k+1)\binom{n+k}{k}}.
\endaligned
\end{equation}
Finally, substituting \eqref{e:27} for the second sum \eqref{e:26}, and differentiating the right member of  \eqref{e:25}, with respect to $x$,
   setting $x=n$, and matching the two expressions, we may obtain the desired result \eqref{Thm27-eq}.
\end{proof}

\vskip 3mm
\begin{cor}\label{cor2.8}
Let $p \in \mathbb{C} \setminus \mathbb{Z}_{\leqslant 0}$ and  $m \in \mathbb{Z}_{\geqslant 0}$. Then
\begin{align}\label{Cr28-eq}
\sum_{k=1}^{\infty}\frac{1}{k(p+k)^{m+1}}=\frac{\gamma}{p^{m+1}}+\frac{1}{p^{m+1}}\sum_{j=0}^{m}\frac{(-1)^{j}p^j}{j!}\psi^{(j)}(p+1).
\end{align}
\end{cor}

\begin{proof}
 Setting $n=0$ in \eqref{Thm27-eq} can yield \eqref{Cr28-eq}.
\end{proof}

\vskip 3mm
\begin{thm}\label{Theorem2.9}
Let $p \in \mathbb{C} \setminus \mathbb{Z}_{\leqslant 0}$ and  $m,\,n \in \mathbb{Z}_{\geqslant 0}$. Then
\begin{equation}\label{Th29-eq}
 \aligned
&\sum_{k=1}^{\infty}\frac{(-1)^nH_{k-1}}{k(p+n+k)^{m+1}\binom{n+k}{k}}\\
&\hskip 3mm =\frac{1}{2}\sum_{k=0}^{n}\frac{(-1)^k}{(p+k)^{m+1}}\binom{n}{k}\big\{H_n^2+H_n^{(2)}-H_{n-k}^2-H_{n-k}^{(2)}\big\}\\
&\hskip 6mm +\frac{(-1)^{m}}{2\, m!}\frac{\partial^{m}}{\partial s^{m}}\bigg[\bigg(\frac{\partial^2}{\partial x^2}-2H_n\frac{\partial}{\partial x}\bigg)\frac{\Gamma(x+1)\Gamma(s)}{\Gamma(x+s+1)}\bigg]_{\begin{subarray}{l}
       x=n\\
       s=p
      \end{subarray}}.
 \endaligned
\end{equation}
\end{thm}

\begin{proof} A similar technique of the proof of Theorem \ref{Thm2.1} may be used.
             The details are omitted.
\end{proof}

\vskip 3mm
\begin{cor}\label{Cor2.10}
Let $p \in \mathbb{C} \setminus \mathbb{Z}_{\leqslant 0}$ and  $m \in \mathbb{Z}_{\geqslant 0}$. Then
\begin{equation}\label{Cr2-10-eq}
  \aligned
  & \sum_{k=1}^{\infty}\frac{H_{k-1}}{k(p+k)^{m+1}}
    = \frac{1}{2} \Big\{ (\gamma +\psi (p+1))^2 +\zeta (2)- \psi'(p+1)\Big\}\\
  &\hskip 15mm + \frac{1}{2}\sum_{\ell=1}^{m}\,\frac{(-1)^\ell}{\ell ! \,p^{m-\ell +1}}
    \Big\{2 \gamma\, \psi^{(\ell)}(p+1) -\psi^{(\ell+1)}(p+1) \\
   &\hskip 37mm  + \sum_{j=0}^{\ell}\binom{\ell}{j} \,\psi^{(j)}(p+1)\psi^{(\ell-j)}(p+1) \Big\}.
  \endaligned
\end{equation}
\end{cor}

\begin{proof}
Setting $n=0$ in \eqref{Th29-eq} offers
\begin{equation*}
\sum_{k=1}^{\infty}\frac{H_{k-1}}{k(p+k)^{m+1}}=\frac{(-1)^{m}}{2\,m!}
\left[\frac{\partial^{m}}{\partial s^{m}} \bigg\{ \frac{\partial^2}{\partial x^2}\frac{\Gamma(x+1)\Gamma(s)}{\Gamma(x+s+1)}\bigg\}\bigg|_{x=0}
  \right]\bigg|_{s=p}.
\end{equation*}
Also we may have
 \begin{equation*}
\Gamma(s)\,\frac{\partial^2}{\partial x^2}\frac{\Gamma (x+1)}{\Gamma (x+s+1)}\bigg|_{x=0} = \frac{h(s)}{s},
 \end{equation*}
where
  \begin{equation*}
    h(s):= \gamma^2 + \zeta (2) +2 \gamma\, \psi (s+1) + (\psi (s+1))^2 -\psi' (s+1).
  \end{equation*}
Then one may get the desired identity \eqref{Cr2-10-eq}.
\end{proof}

\vskip 3mm
\begin{thm}\label{Theorem2.11}
Let $p \in \mathbb{C} \setminus \mathbb{Z}_{\leqslant 0}$,   $m \in \mathbb{N}$ and $n \in \mathbb{Z}_{\geqslant 0}$. Then
\begin{equation}\label{Th211-eq}
\aligned
&\sum_{k=1}^{\infty}\frac{H_{k-1}^2-H_{k-1}^{(2)}}{k(p+n+k)^{m}\binom{n+k}{k}}
  =\frac{(-1)^{n}}{3}\sum_{k=0}^{n}\frac{(-1)^{k}}{(p+k)^m}\binom{n}{k}\\
&\times\big\{H_n^3+2H_n^{(3)} +3H_nH_n^{(2)}-H_{n-k}^3-2H_{n-k}^{(3)}-3H_{n-k}H_{n-k}^{(2)}\big\}\\
&\hskip 5mm +\frac{(-1)^{m+n}\big(H_n^2+H_n^{(2)}\big)}{(m-1)!}\frac{\partial}{\partial x}\bigg[\frac{\partial^m}{\partial z^m}\frac{\Gamma(x+1)\Gamma(z)}{\Gamma(x+z+1)}\bigg]_{\begin{subarray}{l}
       x=n\\
       z=p
      \end{subarray}}\\
&\hskip 5mm-\frac{(-1)^{m+n}H_n}{(m-1)!}\frac{\partial^2}{\partial x^2}\bigg[\frac{\partial^m}{\partial z^m}\frac{\Gamma(x+1)\Gamma(z)}{\Gamma(x+z+1)}\bigg]_{\begin{subarray}{l}
       x=n\\
       z=p
      \end{subarray}}\\
&\hskip 5mm+\frac{(-1)^{m+n}}{3(m-1)!}\frac{\partial^3}{\partial x^3}\bigg[\frac{\partial^m}{\partial z^m}\frac{\Gamma(x+1)\Gamma(z)}{\Gamma(x+z+1)}\bigg]_{\begin{subarray}{l}
       x=n\\
       z=p
      \end{subarray}}.
\endaligned
\end{equation}
\end{thm}

\begin{proof}
The proof would  parallel  that of Theorem \ref{Theorem2.3}.  The details are omitted.
\end{proof}

\vskip 3mm
\begin{cor} \label{Cor2.12}
Let $p \in \mathbb{C} \setminus \mathbb{Z}_{\leqslant 0}$ and    $m \in \mathbb{N}$. Then
\begin{equation}\label{Cor2.12-eq}
 \aligned
 &\sum_{k=1}^{\infty}\frac{H_{k-1}^2-H_{k-1}^{(2)}}{k(p+k)^{m}}=\sum_{\ell=0}^{m}\,\binom{m}{\ell}\, (-1)^{m-\ell}\, \frac{(m-\ell)!}{z^{m-\ell+1}}\,g^{(\ell)}(p),
  \endaligned
\end{equation}
where
\begin{equation*}
  \aligned
g(z) = &-\gamma^3 -3\gamma\,\zeta (2) -2 \zeta (3) - 3\{ \gamma^2+ \zeta (2)\}\,\psi (z+1)\\
      &+3 \gamma\,\psi' (z+1)- \psi^{(2)}(z+1)\\
      & + 3 \psi (z+1)\,\psi' (z+1)-3 \gamma\,\psi^2 (z+1)-\psi^3 (z+1)
  \endaligned
\end{equation*}
and, for $\ell \in \mathbb{N}$,
\begin{equation*}
  \aligned
g^{(\ell)}(z)=& - 3\{ \gamma^2+ \zeta (2)\}\,\psi^{(\ell)} (z+1)+3 \gamma\,\psi^{(\ell +1)} (z+1)- \psi^{(\ell+2)}(z+1)\\
     &+3\, \sum_{j=0}^{\ell}\, \binom{\ell}{j}\, \psi^{(j +1)} (z+1)\,\psi^{(\ell-j)} (z+1)\\
     &-3\gamma\, \sum_{j=0}^{\ell}\, \binom{\ell}{j}\, \psi^{(j)} (z+1)\,\psi^{(\ell-j)} (z+1)\\
     &- \sum_{k=0}^{\ell}\binom{\ell}{k}\,\left\{\sum_{j=0}^{k}\binom{k}{j}\, \psi^{(j)} (z+1)\,\psi^{(k-j)} (z+1) \right\}
        \psi^{(\ell-k)} (z+1).
   \endaligned
\end{equation*}
\end{cor}

\begin{proof}
  Setting $n=0$ \eqref{Th211-eq} gives
\begin{align*}
\sum_{k=1}^{\infty}\frac{H_{k-1}^2-H_{k-1}^{(2)}}{k(p+k)^{m}}=\frac{(-1)^{m}}{3(m-1)!}\bigg[\frac{d^{m}}{d z^{m}}\bigg\{\frac{\partial^3}{\partial x^3}\frac{\Gamma(x+1)\Gamma(z)}{\Gamma(x+z+1)}\bigg|_{x=0}\bigg\}\bigg]_{z=p}.
\end{align*}
\end{proof}

\section{Particular cases and remarks}\label{PR}

This section demonstrates certain particular instances of our main findings
along with pertinent comments (if any).

\vskip 3mm

\textbf{Example 1.} From \eqref{Cor2.4-eq-a}, one may find that
   \begin{equation}\label{PR-ex1-a}
     \sum_{k=1}^{\infty}\frac{H_k^2-H_k^{(2)}}{k^{m+1}}= \mathtt{S}(1^2; m+1)-\mathtt{S}(2, m+1) \quad (m \in \mathbb{N}),
   \end{equation}
which is a combination of a nonlinear harmonic sum and a linear harmonic sum and can be evaluated in terms of
 Riemann zeta functions for any $m \in \mathbb{N}$. The simple one is
 \begin{equation}\label{PR-ex1-b}
     \sum_{k=1}^{\infty}\frac{H_k^2-H_k^{(2)}}{k^{2}}= \mathtt{S}(1^2; 2)-\mathtt{S}(2, 2)
      = \frac{5}{2}\,\zeta (4).
   \end{equation}
As noted in the paragraph between \eqref{Euler-sum} and \eqref{Euler-sum-Symm},
the linear Euler sum $\mathtt{S}(2, m+1)$
is determined in terms of Riemann zeta functions
only when $m=1$, $m=5$, and $m$ is even. So is   $\mathtt{S}(1^2; m+1)$ for $m=1$, $m=5$, and $m$ is even.
For example,
\begin{equation}\label{PR-ex1-c}
 \mathtt{S}(2, 2) = \frac{7}{4}\, \zeta (4)
\end{equation}
and
\begin{equation}\label{PR-ex1-d}
 \mathtt{S}(1^2; 2)=\frac{17}{4}\, \zeta (4).
\end{equation}
The identity in \eqref{PR-ex1-d} was observed by E. Au-Yeung.
The $\mathtt{S}(1^2; m+1)$ for $m=1$, $m=5$, and $m$ is even
was evaluated in terms of Riemann zeta functions by Borwein et al. \cite{Borwein}
who used the Eulerian beta integral in \eqref{beta} and, also,  by Flajolet and  Salvy   \cite{FlajSalv}
who applied  residue calculus to $\psi$ expansions  such as \eqref{LE-psi} and \eqref{LE-polygamm}.

\vskip 3mm
\textbf{Example 2.}
Setting $p=1$ and $m=0$ in \eqref{Cor26-eq} yields an interesting sum
 which involves harmonic numbers and binomial coefficients:
\begin{equation}\label{ER-ex2-a}
  \sum_{k=1}^{\infty}\frac{\left(2H_{2k}-H_k\right)\binom{2k}{k}}{(k+1)\,4^{k+1}}=1.
\end{equation}
Putting $m=1$ in  \eqref{Cor26-eq} offers
\begin{equation}\label{ER-ex2-b}
\aligned
 & \sum_{k=1}^{\infty}\frac{\left(2H_{2k}-H_k\right)\binom{2k}{k}}{(p+k)^2\,4^{k}}
    = \frac{\Gamma (\tfrac{1}{2}) \Gamma (p)}{\Gamma (p+\tfrac{1}{2})}\\
  &\hskip 10mm \times   \Big[ \{\psi (p)-\psi (p+\tfrac{1}{2})\}\{\psi (\tfrac{1}{2})-\psi (p+\tfrac{1}{2})\}
      -\psi' (p+\tfrac{1}{2}) \Big]
\endaligned
\end{equation}
\begin{equation*}
  \left(p \in \mathbb{C} \setminus \mathbb{Z}_{\leqslant 0},\,\, p \ne \tfrac{1-2k}{2} \,\, (k \in \mathbb{N})  \right).
\end{equation*}
The particular case of \eqref{ER-ex2-b} when $p=\tfrac{1}{2}$ gives
\begin{equation}\label{ER-ex2-b-1/2}
\aligned
 & \sum_{k=1}^{\infty}\frac{\left(2H_{2k}-H_k\right)\binom{2k}{k}}{(k+\tfrac{1}{2})^2\,4^{k}}
    = \pi \left(4\log^22 - \tfrac{\pi^2}{6}  \right).
\endaligned
\end{equation}

Extended parametric harmonic sums involving $H_{qk}$ $(q \in \mathbb{N})$
were investigated in \cite{Sofo-2022}.

\vskip 3mm
\textbf{Example 3.}
Setting $m=0$ in \eqref{Cr28-eq} produces a known identity for the psi function (see, e.g., \cite[p. 24]{Sr-Ch-12}):
\begin{equation}\label{ER-ex3-eq1}
 \sum_{k=1}^{\infty}\, \frac{p}{k(p+k)} = \gamma +\psi (p+1) \quad \left(p \in \mathbb{C} \setminus \mathbb{Z}_{\leqslant -1}\right).
\end{equation}
By applying
  \begin{equation*}
    \frac{1}{k}= \frac{1}{k+p}\,\frac{1}{1-\frac{p}{k+p}}= \sum_{j=0}^{\infty}\, \frac{p^j}{(k+p)^{j+1}} \quad (|p|<|k+p|)
  \end{equation*}
to the left member of \eqref{Cr28-eq}, we may obtain
\begin{equation}\label{ER-ex3-eq2}
  \sum_{k=1}^{\infty}\,\frac{1}{k(k+p)^{m+1}}= \sum_{j=0}^{\infty}\,p^j\, \zeta (m+j+2,p+1) \quad (|p|<|1+p|)
\end{equation}
and
\begin{equation}\label{ER-ex3-eq3}
  \sum_{j=0}^{\infty}\,p^j\, \zeta (m+j+2,p+1)= \frac{\gamma}{p^{m+1}}+\frac{1}{p^{m+1}}\sum_{j=0}^{m}\frac{(-1)^{j}p^j}{j!}\psi^{(j)}(p+1)
\end{equation}
\begin{equation*}
  \left(p \in \mathbb{C} \setminus \mathbb{Z}_{\leqslant 0}, \,\,m \in \mathbb{Z}_{\geqslant 0},\,\,|p|<|1+p|\right).
\end{equation*}
Setting $p=1$ in \eqref{ER-ex3-eq3} provides
\begin{equation}\label{ER-ex3-eq4}
  \sum_{j=2}^{\infty}\,\{\zeta (m+j)-1\}=m+1-\sum_{k=1}^{m}\,\zeta (k+1)
\quad \left(m \in \mathbb{Z}_{\geqslant 0}\right).
\end{equation}
Putting $m=0$ in \eqref{ER-ex3-eq4} offers
\begin{equation}\label{ER-ex3-eq5}
  \sum_{j=2}^{\infty}\,\{\zeta (j)-1\}=1.
\end{equation}
In fact, Shallit and Zikan \cite{Sh-Zi} revealed that
a relatively traditional (more than two centuries old) theorem of
Christian Goldbach (1690--1764), which was given in a letter dated 1729 from
Goldbach to Daniel Bernoulli (1700--1782):
  \begin{equation}\label{GB-Th}
    \sum_{\eta \in E}\,(\eta-1)^{-1} =1 \quad  \big(E:=\big\{n^k \,\big|\,n,\, k \in \mathbb{Z}_{\geqslant 2}\big\}\big)
  \end{equation}
is turned out to be the elegant form \eqref{ER-ex3-eq5}.

  The research topic of series involving zeta functions such as \eqref{ER-ex3-eq3}, \eqref{ER-ex3-eq4}, and \eqref{ER-ex3-eq5}
has been popularly investigated by many researchers who have presented closed form expressions of
a variety of series involving zeta functions and given applications (see, e.g., \cite[Chapter 3]{Sr-Ch-12} and the references cited therein;
for recent ones, see also \cite{Al-Ch-AADM}, \cite{Alzer},  \cite{Choi-ADE}, \cite{Ch-Sr}, \cite{Sofo-2022}).

\vskip 3mm
\textbf{Example 4.} Setting $p=\frac{1}{2}$ in \eqref{Cr2-10-eq} produces
\begin{equation}\label{Cr2-10-eq-a}
  \aligned
  & \sum_{k=2}^{\infty}\frac{H_{k-1}}{k \left(k-\tfrac{1}{2}\right)^{m+1}}
    =2\,\log^22- \zeta (2)\\
  &\hskip 5mm +(-1)^{m+1}\,2^m\,\sum_{\ell=1}^{m}\,(-1)^\ell\,2^{-\ell}\,
    \Big\{(\ell+1)\, (1-2^{\ell +2})\, \zeta (\ell+2) \\
   &\hskip 5mm  +  4\, \log 2\,(2^{\ell +1}-1)\,\zeta (\ell+1)\\
   &\hskip 5mm +  \sum_{j=1}^{\ell-1}\,(2^{j +1}-1)\, (2^{\ell-j +1}-1)\,\zeta (j+1) \, \zeta (\ell-j+1)\Big\}.
  \endaligned
\end{equation}

\bibliographystyle{amsplain}

\begin{thebibliography}{100}



\bibitem{Al-Ch-AADM} H. Alzer and J. Choi,  The Riemann zeta function and classes of
               infinite series, \emph{Appl. Anal. Discrete Math.} \textbf{11} (2017), 386--398.
   \url{https://doi.org/10.2298/AADM1702386A}


\bibitem{Alzer} H. Alzer and J. Choi, Four parametric linear Euler sums, \emph{J. Math. Anal. Appl.} \texttt{484} (1) (2020), ID123661,
      \url{https://doi.org/10.1016/j.jmaa.2019.123661}.


\bibitem{Basu} A. Basu and T. M. Apostol, A new method for investigating Euler sums, \emph{Ramanujan J}. \textbf{4} (2000), 397-419.
        \url{https://doi.org/10.1023/A:1009868016412}


\bibitem{1} N. Bat{\i}r and A. Sofo,  Sums involving the binomial and Gregory coefficients and harmonic numbers, submitted, 2022.


\bibitem{Berndt}B. C. Berndt, Ramanujan's Notebooks, Part I, Springer-Verlag, New York, Berlin,   1985.


\bibitem{Borwein} D. Borwein, J. M. Borwein and R. Girgensohn,  Explicit
evaluation of Euler sums, \emph{Proc. Edinburgh Math. Soc.}
\textbf{38}(2) (1995), 277--294. \url{doi:10.1017/S0013091500019088}



\bibitem{Chavan} P. Chavan and S. Chavan, On explicit evaluation of certain linear alternating Euler sums and double $t$-values,
    \emph{J. Anal.} (2022).  \url{https://doi.org/10.1007/s41478-022-00472-4}.


 \bibitem{Choi-ADE} J. Choi,  Determinants of the Laplacians on the $n$-dimensional unit sphere $\mathbf{S}^n$,
    \emph{Adv. Diff. Equ.} \textbf{2013} (2013), ID 236.
   \url{http://www.advancesindifferenceequations.com/content/2013/1/236}


\bibitem{Choi} J. Choi  and H. M. Srivastava,
     Explicit evaluation of Euler and related sums,
       \emph{Ramanujan J.} \textbf{10} (2005), 51--70.
              \url{https://doi.org/10.1007/s11139-005-3505-6}

\bibitem{Ch-Sr} J. Choi and H. M. Srivastava,  Series involving the Zeta functions and a family of generalized
                     Goldbach-Euler Series, \emph{Amer. Math. Monthly} \textbf{121} (2014),
                      229--236.
                     \url{http://dx.doi.org/10.4169/amer.math.monthly.121.03.229}

\bibitem{FlajSalv} P. Flajolet and B. Salvy, Euler sums and contour integral
representations, \emph{Exp. Math.} \textbf{7}(1) (1998), 15--35.
\url{https://doi.org/10.1080/10586458.1998.10504356}

\bibitem{Freitas} P. Freitas, Integrals of polylogarithmic functions,
             recurrence relations, and associated Euler sums,  \emph{Math. Comput.} \textbf{74}(251) (2005), 1425--1440.


\bibitem{Foll} G. B. Folland, \emph{REAL ANALYSIS}, \emph{Modern Techniques and Their Applications},
         John Wiley $\&$ Sons, Inc.,  New York, Chichester, Brisbane, Toronto, Singapore, 1984.

\bibitem{Magn} W. Magnus,  F. Oberhettinger  and R. P. Soni, \emph{Formulas and Theorems for the Special Functions of Mathematical Physics}, Third enlarged Edition, Springer-Verlag, New York, 1966.

\bibitem{Niel-65}
N. Nielsen, {\it  Die Gammafunktion},
Chelsea Publishing Company, Bronx, New York, 1965.

\bibitem{Quan} J. Quan, C. Xu and X. Zhang, Some evaluations of parametric Euler type sums of harmonic numbers,
\emph{Integral Transforms Spec. Funct.} (2022).  \url{https://doi.org/10.1080/10652469.2022.2097671}.

\bibitem{Sh-Zi} J. D. Shallit and K. Zikan, A theorem of Goldbach, \emph{Amer. Math. Monthly} \textbf{93} (1986), 402--403.


\bibitem{Sofo-2022} A. Sofo and J. Choi, Extension of the four Euler sums being linear with parameters and series involving the zeta functions,
      \emph{J. Math. Anal. Appl.} \textbf{515}(1) (2022), ID126370. \url{https://doi.org/10.1016/j.jmaa.2022.126370}


\bibitem{Sofo-2012} A. Sofo and D. Cvijovi\'{c}, Extensions of Euler
harmonic sums,  \emph{Appl. Anal. Discrete Math.} \textbf{6}(2) (2012),  317--328.
\url{doi:10.2298/AADM120628016S}

\bibitem{Sofo-2019} A. Sofo and A. S. Nimbran,  Euler sums and integral
       connections, \emph{Mathematics} \textbf{7} (2019), Article ID 833.
          \url{https://doi.org/10.3390/math7090833}




\bibitem{SofoSri} A. Sofo and H. M. Srivastava,  A family of shifted
harmonic sums,  \emph{Ramanujan J.} \textbf{37}(1) (2015),  89--108.
\url{https://doi.org/10.1007/s11139-014-9600-9}




\bibitem{Sr-Ch-12} H. M. Srivastava and J. Choi, \emph{Zeta and $q$-Zeta
Functions and Associated Series and Integrals}, Elsevier, Inc., Amsterdam,
2012.




\bibitem{Xu} C. Xu and W. Wang, Two variants of Euler sums, \emph{Monatsh. Math.} (2022).
   \url{https://doi.org/10.1007/s00605-022-01683-4}.





\end{thebibliography}

\end{document}